\documentclass[a4paper,twoside,reqno]{amsart}
\usepackage{amssymb}
\usepackage[mathscr]{eucal}

\usepackage{hyperref}

\theoremstyle{plain}
\newtheorem{prop}{Proposition}[section]
\newtheorem{thm}[prop]{Theorem}
\newtheorem{cor}[prop]{Corollary}
\newtheorem{lem}[prop]{Lemma}

\theoremstyle{definition}
\newtheorem{dfn}[prop]{Definition}
\newtheorem{rem}[prop]{Remark}
\newtheorem{lab}[prop]{}

\newcommand{\C}{{\mathbb{C}}}
\renewcommand{\P}{{\mathbb{P}}}
\newcommand{\R}{{\mathbb{R}}}

\newcommand{\scrF}{{\mathscr{F}}}

\newcommand{\scrU}{{\mathscr{U}}}

\newcommand{\sfS}{\mathsf{S}}

\DeclareMathOperator{\im}{im}
\DeclareMathOperator{\interior}{int}
\DeclareMathOperator{\relint}{relint}
\DeclareMathOperator{\rk}{rk}

\DeclareMathOperator{\spn}{span}

\newcommand{\aff}{\mathrm{aff}}
\newcommand{\Ex}{\mathrm{Ex}}
\newcommand{\Gram}{\mathrm{Gram}}

\newcommand{\To}{\Rightarrow}

\renewcommand{\subset}{\subseteq}
\renewcommand{\supset}{\supseteq}
\newcommand{\idl}[1]{\langle #1\rangle}

\newcommand{\du}{{\scriptscriptstyle\vee}}
\newcommand{\ol}[1]{\overline{#1}}
\newcommand{\plus}{{\scriptscriptstyle+}}

\renewcommand{\emptyset}{\varnothing}
\renewcommand{\setminus}{\smallsetminus}
\renewcommand{\epsilon}{\varepsilon}
\renewcommand{\theta}{\vartheta}

\newcommand{\bil}[2]{\langle{#1},{#2}\rangle}
\renewcommand{\choose}[2]{\genfrac(){0pt}{}{#1}{#2}}

\begin{document}

\title
[Extreme points of Gram spectrahedra of binary forms]
{Extreme points of\\[3pt]Gram spectrahedra of binary forms}

\author
 {Claus Scheiderer}
\address
 {Fachbereich Mathematik und Statistik, Universit\"at Konstanz,
 78457 Konstanz, Germany}
\email
 {claus.scheiderer@uni-konstanz.de}

\begin{abstract}
The Gram spectrahedron $\Gram(f)$ of a form $f$ with real
coefficients parametrizes the sum of squares decompositions of $f$,
modulo orthogonal equivalence. For $f$ a sufficiently general
positive binary form of arbitrary degree, we show that $\Gram(f)$ has
extreme points of all ranks in the Pataki range. This is the first
example of a family of spectrahedra of arbitrarily large dimensions
with this property. We also calculate the dimension of the set of
rank $r$ extreme points, for any~$r$. Moreover, we determine the
pairs of rank two extreme points for which the connecting line
segment is an edge of $\Gram(f)$.
\end{abstract}

\thanks
{Supported by Deutsche Forschungsgemeinschaft (DFG) under grant
\texttt{SCHE281/10}. Part of this work was done in Fall 2017 while
the author enjoyed the hospitality of MSRI Berkeley. Both are
gratefully acknowledged. I am much indebted to Thorsten Mayer for his
careful reading and for finding an error in a previous version.}
\maketitle

%-------------------------------------------------------------------%

\section{Introduction}

Given a form $f$ that is a sum of squares of forms, there are usually
many inequivalent ways of writing $f$ as a sum of squares. The set
$\Gram(f)$ of all sum of squares (\emph{sos}) representations of
$f$, modulo orthogonal equivalence, has a natural structure of a
spectrahedron, so it is an object of geometric nature. Studying the
convex-geometric properties of $\Gram(f)$, and in particular its
extreme points, is relevant for the problem of optimizing linear
functions over all sum of squares representations of~$f$.
With probability one, the optimizer for a random such problem will be
a unique extreme point of $\Gram(f)$.
From an algebraic perspective, studying the extreme points of the
Gram spectrahedron is natural since every sos representation of $f$
arises as a convex combination of representations that correspond to
extreme points of $\Gram(f)$.

Although the basic idea goes back to Choi, Lam and Reznick \cite{clr}
in 1995, a systematic study of Gram spectrahedra was taken up only
recently. Gram spectrahedra of ternary quartics were considered
by Plaumann, Sturmfels and Vinzant in \cite{psv}.
The paper \cite{cpsv} by Chua, Plaumann, Sinn and Vinzant is a survey
of results and open questions on Gram spectrahedra.
Among others, the authors discuss Gram spectrahedra of binary forms,
and for sextic binary forms they relate the Gram spectrahedra to
Kummer surfaces in $\P^3$, see also \cite{orsv}.

Any point of a spectrahedron has a rank. The Pataki interval
describes the range of values that the rank of an extreme point of a
general spectrahedron may have. For points of Gram spectrahedra, the
rank is identified with the length of the corresponding sum of
squares decomposition. In particular, the sum of squares length
of~$f$, or the collection of different sum of squares representations
of a given length, are naturally encoded in $\Gram(f)$. These are
invariants that have received a lot of attention in particular cases,
starting with Hilbert \cite{hi}, and more recently \cite{prss}, for
ternary quartics. Lately, results of a similar spirit were obtained
for varieties of minimal or almost minimal degree, see
\cite{bsv, bpsv, sch:pyth, cpsv}.

In this paper we focus on Gram spectrahedra in the most basic case
possible, namely binary forms.
For $f$ a sufficiently general positive binary form of arbitrary
degree, we show that $\Gram(f)$ has extreme points of all ranks in
the Pataki range (Theorem \ref{main}). This gives a positive answer
to Question~4.2 from \cite{cpsv}. It also establishes the first known
instance of a family of spectrahedra of arbitrary dimensions with
this property.
In fact we calculate the dimension of the set of extreme points of
any given rank $r$, for $f$ sufficiently general (Corollary
\ref{dimexr}).

The proofs for these facts rely on a purely algebraic result of
independent interest (Theorem \ref{quadindepallg}): For any integers
$d\ge0$ and $r\ge1$ with $\choose{r+1}2\le2d+1$, there exists a
sequence $(p_1,\dots,p_r)$ of $r$ binary forms of degree~$d$ for
which the $\choose{r+1}2$ products $p_ip_j$ ($1\le i\le j\le r$) are
linearly independent. Any sequence with this property will be called
\emph{quadratically independent}.

When $f$ is a general positive binary form of degree $2d$,
$\Gram(f)$ has precisely $2^{d-1}$ extreme points of rank~two. Given
two of these points, the line segment connecting them may or may not
be a face (edge) of $\Gram(f)$. For sextic forms we show that it is
never an edge, while for $2d\ge10$ it always is an edge. Most
interesting is the case $\deg(f)=8$, where the edges between the
eight rank two extreme points form a complete bipartite graph
$K_{4,4}$ (Theorem \ref{edges}).

We briefly comment on our methods. Throughout we pursue a
coordinate-free approach to Gram spectrahedra. Let $\theta\in
\Gram(f)$, and let $F$ be the face of $\Gram(f)$ that has $\theta$ in
its relative interior. We constantly use the following
characterization of $\dim(F)$:
If $f=p_1^2+\cdots+p_r^2$ is the sos representation that corresponds
to $\theta$ (with $p_1,\dots,p_r$ linearly independent forms),
$\dim(F)$ is the number of quadratic relations between
$p_1,\dots,p_r$. In particular, $\theta$ is an extreme point of
$\Gram(f)$ if and only if
$p_1,\dots,p_r$ are quadratically independent.

The paper is organized as follows. In Section~2 we review the
well-known results by Ramana and Goldman on the facial structure of
spectrahedra, together with the Pataki range for the rank. We then
specialize to Gram spectrahedra and formulate the dimension formula
for faces in terms of quadratic relations. Section~4 contains the
proof for the existence of long quadratically independent sequences
of binary forms. In Sections~5 and~6 we present our analysis of the
ranks of extreme points and of the edges between rank two extreme
points.

We use standard terminology from convex geometry.
For $K\subset\R^n$ a closed convex set, $\aff(K)$ denotes the
affine-linear hull of $K$ and $\relint(K)$ is the relative interior
of $K$, i.e.\ the interior of $K$ relative to $\aff(K)$. A convex
subset $F\subset K$ is a face of $K$ if $x,\,y\in K$, $0<t<1$ and
$(1-t)x+ty\in F$ imply $x,\,y\in F$. For every $x\in K$ there is a
unique face $F$ of $K$ with $x\in\relint(F)$, called the supporting
face of~$x$.

%-------------------------------------------------------------------%

\section{Review of facial structure of spectrahedra}\label{sectspec}%

All results in this section are known. They are due to Ramana and
Goldman \cite{rg} for the first part, and to Pataki \cite{pa} for the
Pataki range. We nevertheless give them
a coordinate-free review here, i.e.\ without making reference to a
particular basis of the underlying vector space.

\begin{lab}
Let $V$ be a vector space over $\R$ with $\dim(V)<\infty$. Let
$V^\du$ be the dual space of $V$, and let $\sfS_2V\subset V\otimes V$
denote the space of symmetric tensors, i.e.\ tensors that are
invariant under the involution $v\otimes w\mapsto w\otimes v$. Of
course, $\sfS_2V$ is canonically identified with $\sfS^2V$, the
second symmetric power of~$V$, but it seems preferable in our context
to work with $\sfS_2$, rather than with $\sfS^2$.
The natural pairing between $v\in V$ and $\lambda\in V^\du$ is
denoted $\bil v\lambda=\bil\lambda v$. Elements of $\sfS_2V$ can be
identified either with symmetric bilinear forms $V^\du\times V^\du
\to\R$, or with self-adjoint linear maps $\varphi\colon V^\du\to V$,
where the adjoint refers to the natural pairing between $V$ and
$V^\du$.
We shall adapt the second point of view. Let $\varphi_\theta\colon
V^\du\to V$ denote the
linear map that corresponds to a
symmetric tensor $\theta=\sum_{i=1}^rv_i\otimes w_i\in\sfS_2V$. So
$\varphi_\theta(\lambda)=\sum_{i=1}^r\lambda(v_i)w_i=\sum_{i=1}^r
\lambda(w_i)v_i$ for $\lambda\in V^\du$.
The \emph{range} of $\theta\in\sfS_2V$, written $\im(\theta)$, is the
range (image) of the linear map $\varphi_\theta$. Thus, if
$v_1,\dots,v_r$ and $w_1,\dots,w_r$ are linearly independent,
$\im(\theta)=\spn(v_1,\dots,v_r)=\spn(w_1,\dots,w_r)$. The
\emph{rank} of $\theta$ is $\rk(\theta)=\dim\im(\theta)$.
\end{lab}

\begin{lab}
$\theta\in\sfS_2V$ is \emph{positive semidefinite} (\emph{psd}),
written $\theta\succeq0$, if $\bil{\varphi_\theta(\lambda)}\lambda
\ge0$ for every $\lambda\in V^\du$. If $\theta=\sum_iv_i\otimes
w_i$, this says $\sum_i\lambda(v_i)\lambda(w_i)\ge0$ for every
$\lambda\in V^\du$. The set $\sfS_2^\plus V=\{\theta\in\sfS_2V\colon
\theta\succeq0\}$ is a closed convex cone in $\sfS_2V$. If
$v_1,\dots,v_n\in V$ are linearly independent and $\theta=
\sum_{i=1}^na_{ij}v_i\otimes v_j$, where
$a_{ij}=a_{ji}\in\R$, then $\theta\succeq0$ if and only if the real
symmetric matrix $(a_{ij})$ is psd, i.e.\ has nonnegative
eigenvalues.
So $\sfS_2^\plus V$ gets identified with the cone of real symmetric
psd $n\times n$-matrices ($n=\dim(V)$), after fixing a linear basis
of $V$. We say that $\theta\in\sfS_2V$ is \emph{positive definite},
written $\theta\succ0$, if $\bil{\varphi_\theta(\lambda)}\lambda>0$
for every $0\ne\lambda\in V^\du$.

The fact that every real symmetric matrix can be diagonalized implies
that every $\theta\in\sfS_2V$ can be written $\theta=\sum_{i=1}^r
\epsilon_iv_i\otimes v_i$, with $r\ge0$, $\epsilon_i=\pm1$ and with
$v_1,\dots,v_r\in V$ linearly independent. Of course,
$\theta\succeq0$ is equivalent to $\epsilon_1=\cdots=\epsilon_r=1$.
\end{lab}

\begin{lem}\label{verwende}%
Given $\theta\in\sfS_2V$ and a linear subspace $U\subset V$, we have
$\im(\theta)\subset U$ if and only if $\theta\in\sfS_2U$.
\end{lem}

\begin{proof}
The ``if'' direction is clear. Conversely assume $\im(\theta)\subset
U$, and write $\theta=\sum_{i=1}^rc_iv_i\otimes v_i$ with $0\ne
c_i\in\R$ and $v_1,\dots,v_r\in V$ linearly independent. If
$\lambda_1,\dots,\lambda_r\in V^\du$ are chosen with
$\bil{v_i}{\lambda_j}=\delta_{ij}$ for all $i,j$, we have
$\varphi_\theta(\lambda_j)=\sum_ic_i\lambda_j(v_i)v_i=c_jv_j$, and by
assumption this element lies in $U$ for every $j$.
Therefore $\theta\in\sfS_2U$.
\end{proof}

\begin{lem}\label{rangesum}%
If $\theta,\,\theta'\in\sfS_2V$ are psd, then $\im(\theta+\theta')=
\im(\theta)+\im(\theta')$.
\end{lem}

\begin{proof}
This translates into the well-known fact that, for any two symmetric
psd matrices $A,\,B$, one has $\im(A+B)=\im(A)+\im(B)$.
\end{proof}

\begin{lem}\label{epsabzin}%
Let $\theta,\,\gamma\in\sfS_2V$ with $\theta\succeq0$ and
$\im(\gamma)\subset\im(\theta)$. Then there is a real number
$\epsilon>0$ with $\theta-\epsilon\gamma\succeq0$.
\end{lem}

\begin{proof}
This translates into the following well-known fact about real
symmetric matrices: If $A,\,B$ are such matrices with $\im(B)\subset
\im(A)$, and if $A\succeq0$, there is $\epsilon>0$ with $A-\epsilon B
\succeq0$.
\end{proof}

\begin{lab}\label{facespaces}%
For the following we fix an affine-linear subspace $L\subset\sfS_2V$
together with the corresponding spectrahedron $S=L\cap\sfS_2^\plus
V$. Results \ref{faceslem}--\ref{dimfacaff} below are all due to
Ramana-Goldman \cite{rg}.
For any subset $T\subset S$ we consider the linear subspace
$$\scrU(T)\>:=\>\sum_{\theta\in T}\im(\theta)$$
of $V$. For any linear subspace $U\subset V$, the set
$$\scrF(U)\>:=\>\{\theta\in S\colon\im(\theta)\subset U\}\>=\>
L\cap\sfS_2^\plus U$$
(Lemma \ref{verwende}) is a face of $S$ by Lemma \ref{rangesum}.
\end{lab}

\begin{lem}\label{faceslem}%
For any face $F\ne\emptyset$ of $S$ there is a linear subspace
$U\subset V$ with $F=\scrF(U)$. In fact we may take $U=\scrU(F)$.
\end{lem}

\begin{proof}
The inclusion $F\subset\scrF(\scrU(F))$ is trivial.
Conversely there exist finitely many $\theta_1,\dots,\theta_m\in F$
with $\scrU(F)=\sum_{i=1}^m\im(\theta_i)$.
Hence there exists a single $\theta\in F$ with $\scrU(F)=
\im(\theta)$, e.g.\ $\theta=\frac1m\sum_{i=1}^m\theta_i$ (Lemma
\ref{rangesum}). In order to prove $\scrF(\im(\theta))\subset F$ let
$\gamma\in\scrF(\im(\theta))$, so $\gamma\in S$ and $\im(\gamma)
\subset\im(\theta)$. Choose a real number $t>0$ so that $\theta':=
\theta-t(\gamma-\theta)\succeq0$, using Lemma \ref{epsabzin}.
Since $\theta'\in S$ and $\theta$ is a convex combination of
$\theta'$ and $\gamma$, we conclude that $\gamma\in F$.
\end{proof}

\begin{dfn}\label{dfnsat}%
We say that a linear subspace $U$ of $V$ is \emph{facial}, or a
\emph{face subspace} (for the given spectrahedron
$S=L\cap\sfS_2^\plus V$), if there exists $\theta\in S$ with
$U=\im(\theta)$.
\end{dfn}

The following lemma is obvious (cf.\ \ref{rangesum}):

\begin{lem}
If $U,\,U'\subset V$ are face subspaces for $S$ then so is their
sum $U+U'$.
\qed
\end{lem}

Note that the intersection $U\cap U'$ need not contain any face
subspace.

\begin{prop}
There is a natural inclusion-preserving bijection between the
nonempty faces $F$ of $S$ and the face subspaces $U\subset V$ for
$S$, given by $F\mapsto\scrU(F)$. The inverse is $U\mapsto\scrF(U)$.
\end{prop}

\begin{proof}
Let $F\ne\emptyset$ be a face of $S$. As in the proof of Lemma
\ref{faceslem}, there is $\theta\in F$ with $\im(\theta)=\scrU(F)$.
Hence the subspace $\scrU(F)$ of $V$ is facial, and
$F=\scrF(\scrU(F))$ holds by \ref{faceslem}. On the other hand, if
$U\subset V$ is a face subspace then $U=\scrU(\scrF(U))$ holds.
Indeed, $\supset$ is tautologically true. Conversely there is
$\theta\in S$ with $U=\im(\theta)$, since $U$ is facial, so we have
$\theta\in\scrF(U)$
and therefore $U\subset\scrU(\scrF(U))=\sum_{\gamma\in\scrF(U)}
\im(\gamma)$.
\end{proof}

In particular we see:

\begin{cor}\label{relint}%
If $U\subset V$ is a face subspace, the relative interior of
$\scrF(U)$ is $\{\theta\in S\colon\im(\theta)=U\}$. The supporting
face of $\theta\in S$ is $\scrF(\im(\theta))$.
\qed
\end{cor}

\begin{cor}
Let $F$ be a face of $S$. Then $\rk(\theta)=\dim\scrU(F)$ for every
$\theta\in\relint(F)$. We call this number the \emph{rank} of~$F$,
denoted $\rk(F)$. If $F'$ is a proper subface of $F$ then
$\rk(F')<\rk(F)$.
\qed
\end{cor}

Here are equivalent characterizations of face subspaces:

\begin{prop}\label{chaktsatur}%
For a linear subspace $U\subset V$, the following are equivalent:
\begin{itemize}
\item[(i)]
$U$ is facial, i.e.\ there is $\theta\in S$ with $\im(\theta)=U$;
\item[(ii)]
$U$ has a linear basis $u_1,\dots,u_r$ for which $\sum_{i=1}^r
u_i\otimes u_i\in S$;
\item[(iii)]
$U$ is linearly spanned by vectors $u_1,\dots,u_r$ for which
$\sum_{i=1}^ru_i\otimes u_i\in S$;
\item[(iv)]
for every $u\in U$ there are $\epsilon>0$ and $u_2,\dots,u_r\in U$
such that $\epsilon u\otimes u+\sum_{i=2}^ru_i\otimes u_i\in S$.
\end{itemize}
\end{prop}

\begin{proof}
(i) $\To$ (ii):
Let $\theta\in S$ with $\im(\theta)=U$. By Lemma \ref{verwende} we
can write $\theta=\sum_{i=1}^ru_i\otimes u_i$ where $u_1,\dots,u_r\in
U$ are linearly independent. Since the $u_i$ span $\im(\theta)$, they
are a linear basis of $U$.

(ii) $\To$ (iii) is trivial.

(iii) $\To$ (iv):
Let $\theta=\sum_{i=1}^ru_i\otimes u_i\in S$ as in (iii). Then
$\im(\theta)=\spn(u_1,\dots,u_r)=U$ by Lemma \ref{rangesum}.
Given $u\in U$ there exists $\epsilon>0$ such that $\gamma:=
\theta-\epsilon u\otimes u\in\sfS_2^\plus U$ (Lemma \ref{epsabzin}).
Hence there exist $u_2,\dots,u_r\in U$ with $\gamma=\sum_{i=2}^r
u_i\otimes u_i$.

(iv) $\To$ (i):
Let $u\in U$, and let $\gamma=\epsilon u\otimes u+\sum_{i=2}^r
u_i\otimes u_i$ be as in (iv). Then $u\in\im(\gamma)$ (Lemma
\ref{rangesum}),
and $\im(\gamma)\subset U$. This shows that there is a (finite)
family of tensors $\gamma_j\in\scrF(U)$ with $\sum_j\im(\gamma_j)=U$.
Hence $U$ is facial.
\end{proof}

\begin{prop}\label{dimfacaff}%
Let $S=L\cap\sfS_2^\plus V$, with $L\subset\sfS_2V$ an affine-linear
subspace. If $F$ is a nonempty face of $S$ and $U=\scrU(F)$, then
$\aff(F)=L\cap\sfS_2U$. In particular, $\dim(F)=\dim(L\cap\sfS_2U)$.
\end{prop}

\begin{proof}
Here $\aff(F)$ denotes the affine-linear hull of $F$. Since
$F=L\cap\sfS_2^\plus U$,
it is clear that $\aff(F)\subset L\cap\sfS_2U$. For the other
inclusion let $\theta\in\relint(F)$, so $\im(\theta)=U$
(\ref{relint}), and let $\gamma\in L\cap\sfS_2U$ be arbitrary. Then
$\gamma_t:=(1-t)\theta+t\gamma\succeq0$ for $|t|<\epsilon$ and small
$\epsilon>0$ (\ref{epsabzin}), and therefore $\gamma_t\in S$ for
these $t$.
Since $\theta=\frac12(\gamma_t+\gamma_{-t})$, these $\gamma_t$ lie in
$F$, and we have proved $\gamma\in\aff(F)$.
\end{proof}

The following result is due to Pataki \cite{pa}. It describes the
interval in which the ranks of the extreme points of a spectrahedron
can possibly lie:

\begin{prop}\label{patineq}%
\emph{(Pataki inequalities})
Let $\dim(V)=n$, let $L\subset\sfS_2V$ be an affine subspace with
$\dim(L)=m$, and let $S=L\cap\sfS_2^\plus V$.
\begin{itemize}
\item[(a)]
For every extreme point $\theta$ of $S$, the rank $\rk(\theta)=r$
satisfies
$$m+\choose{r+1}2\>\le\>\choose{n+1}2.$$
\item[(b)]
When $L$ is chosen generically among all affine subspaces of
dimension $m$, every $\theta\in L\cap\sfS_2^\plus V$ satisfies
$m\ge\choose{n-\rk(\theta)+1}2$.
\end{itemize}
\end{prop}

This formulation is taken from \cite{cpsv} Proposition 3.1. See also
\cite{pa} Corollary 3.3.4
and \cite{nrs} Proposition~5.

\begin{rem}\label{dfnpatrang}%
Let $S=L\cap\sfS_2^\plus V$, where $\dim(V)=n$ and $L\subset\sfS_2V$
is a nonempty affine subspace, $\dim(L)=m$.
The \emph{Pataki interval} for the rank $r$ of extreme points of $S$
is described by the inequalities
\begin{equation}\label{patakineqs}%
m\>\ge\>\choose{n-r+1}2\text{ \ and \ }m+\choose{r+1}2\>\le\>
\choose{n+1}2
\end{equation}
from Proposition \ref{patineq}. This amounts to the range of integers
$r$ satisfying
$$n+\frac12-\frac12\sqrt{8m+1}\>\le\>r\>\le\>-\frac12+\frac12
\sqrt{(2n+1)^2-8m}.$$
Indeed, the first (resp.\ second) inequality in \eqref{patakineqs}
says $A_1\le r\le A_2$ (resp.\ $B_1\le r\le B_2$) where
$$A_i\>=\>n+\frac12+\frac{(-1)^i}2\sqrt{8m+1},\quad B_i\>=\>
-\frac12+\frac{(-1)^i}2\sqrt{(2n+1)^2-8m}$$
($i=1,2$).
It is elementary to check that $B_1<0<A_1\le B_2<A_2$ holds.
Therefore the Pataki interval is $\lceil A_1\rceil\le r\le
\lfloor B_2\rfloor$.
\end{rem}

%-------------------------------------------------------------------%

\section{Gram spectrahedra}

See Choi-Lam-Reznick \cite{clr} for an introduction to Gram matrices
of real polynomials, and Chua-Plaumann-Sinn-Vinzant \cite{cpsv} for a
survey on Gram spectrahedra. In contrast to these texts we emphasize
a coordinate-free approach.

\begin{lab}
Let $A$ be an $\R$-algebra. The multiplication map $A\otimes
A\to A$, $(a,b)\mapsto ab$ (with $\otimes=\otimes_\R$ always)
induces the $\R$-linear map $\mu\colon\sfS_2A\to A$, where
$\sfS_2A\subset A\otimes A$ is the space of symmetric tensors as in
Section~\ref{sectspec}. Given $f\in A$, the symmetric tensors
$\theta\in\sfS_2A$ with $\mu(\theta)=f$ are called the \emph{Gram
tensors} of~$f$.
\end{lab}

\begin{lab}\label{dfngramspec}%
Let $V\subset A$ be a finite-dimensional linear subspace, and let
$f\in A$. We define the \emph{Gram spectrahedron} of $f$, relative
to $V$, to be the set of all psd Gram tensors of $f$ in $\sfS_2V$,
i.e.
$$\Gram_V(f)\>:=\>\sfS_2^\plus V\cap\mu^{-1}(f).$$
It is well-known that $\Gram_V(f)$ parametrizes the sums of squares
representations $f=\sum_{i=1}^rp_i^2$ with $p_i\in V$ for all~$i$,
up to orthogonal equivalence. This means, the elements of
$\Gram_V(f)$ are the symmetric tensors $\sum_{i=1}^rp_i\otimes
p_i$ with $r\ge0$ and $p_1,\dots,p_r\in V$ such that $\sum_{i=1}^r
p_i^2=f$. Given two such tensors $\theta=\sum_{i=1}^rp_i\otimes p_i$
and $\theta'=\sum_{j=1}^sq_j\otimes q_j$, we may assume $r=s$; then
$\theta=\theta'$ if and only if there is an orthogonal real matrix
$(u_{ij})$ such that $q_j=\sum_{i=1}^ru_{ij}p_i$ for all $j$. See
\cite{clr} \S\,2.
\end{lab}

\begin{lem}
$\Gram_V(f)$ is a spectrahedron, and is compact provided that the
identity $\sum_{i=1}^rp_i^2=0$ with $p_1,\dots,p_r\in V$ implies
$p_1=\cdots=p_r=0$.
\end{lem}

\begin{proof}
By its definition, $\Gram_V(f)$ is a spectahedron. If $\Gram_V(f)$ is
unbounded, it has nonzero recession cone, which means that there is
$0\ne\theta\in\sfS_2V$ with $\eta+\theta\in\Gram_V(f)$ for every
$\eta\in\Gram_V(f)$. It follows that $\mu(\theta)=0$ and $\theta
\succeq0$, so $\theta=\sum_{i=1}^rp_i\otimes p_i$ with $0\ne p_i\in
V$ where $\sum_{i=1}^rp_i^2=0$.
\end{proof}

\begin{lab}
For $U\subset A$ a linear subspace let $\Sigma U^2=\{\sum_{i=1}^r
u_i^2\colon r\ge1$, $u_i\in U\}$. Usually we will consider Gram
spectrahedra only in the case where sums of squares in $A$ are
strongly stable \cite{ne}. This means that there exists a filtration
$U_1\subset U_2\subset\cdots\subset\bigcup_{i\ge1}U_i=A$ by
finite-dimensional linear subspaces $U_i$ such that for every $i\ge1$
there is $j\ge1$ with $U_i\cap\Sigma A^2\subset\Sigma U_j^2$. In this
case we simply write $\Gram(f):=\Gram_{U_j}(f)$ for $f\in U_i$.
Examples are the polynomial rings $A=\R[x_1,\dots,x_n]=\R[x]$ with
$U_i=\R[x]_{\le i}$, the space of polynomials of degree $\le i$.
\end{lab}

\begin{lab}\label{uu}%
We summarize what the formalism of Section \ref{sectspec} means.
Let $V\subset A$ be a linear subspace, $\dim(V)<\infty$, and let
$f\in A$. We will say that a linear subspace $U\subset V$ is a
\emph{face subspace for}~$f$ if $U$ is a face space for the
spectrahedron $\Gram_V(f)$ in the sense of \ref{dfnsat}. In other
words, $U$ is a face subspace for $f$ if there is $\theta\in
\Gram_V(f)$ with $U=\im(\theta)$. According to Proposition
\ref{chaktsatur}, the nonempty faces $F$ of $\Gram_V(f)$ are in
bijection with the face subspaces $U$ for $f$, via
$F\mapsto\scrU(F)$ and $U\mapsto\scrF(U)$.

The dimension formula \ref{dimfacaff} for faces takes a particularly
appealing form for Gram spectrahedra. If $U\subset A$ is a linear
subspace, let $UU$ denote the linear subspace of $A$ spanned by the
products $pp'$ ($p,\,p'\in U$).
\end{lab}

\begin{prop}
For $U\subset V$ a face subspace for $f$, the face $\scrF(U)$ of
$\Gram_V(f)$ has dimension
$$\dim\scrF(U)\>=\>\frac r2(r+1)-s$$
with $r=\dim(U)$ and $s=\dim(UU)$.
\end{prop}

\begin{proof}
By Proposition \ref{dimfacaff}, $\dim\scrF(U)$ is the dimension of
the affine space $\mu^{-1}(f)\cap\sfS_2U$. Hence $\dim\scrF(U)=
\dim(W)$ where $W$ is the kernel of the surjective linear map
$\mu\colon\sfS_2U\to UU$.
Since $\dim(W)=\dim(\sfS_2U)-\dim(UU)=\frac r2(r+1)-s$, the
proposition follows.
\end{proof}

\begin{cor}
Let $f=\sum_{i=1}^rp_i^2$ with $p_1,\dots,p_r\in V$ linearly
independent, let $\theta=\sum_{i=1}^rp_i\otimes p_i$ be the
corresponding Gram tensor of~$f$. The dimension of the supporting
face of $\theta$ in $\Gram_V(f)$ equals the number of independent
linear relations between the products $p_ip_j$ ($1\le i\le j\le r$).
\qed
\end{cor}

We say that a sequence $p_1,\dots,p_r$ in $A$ is \emph{quadratically
independent} if the $\choose{r+1}2$ products $p_ip_j$ ($1\le i\le
j\le r$) are linearly independent. Using this terminology we get:

\begin{cor}\label{chaktextpts}%
A psd Gram tensor $\sum_{i=1}^rp_i\otimes p_i$ of $f$, with
$p_1,\dots,p_r\in V$ linearly independent, is an extreme point of
$\Gram_V(f)$ if and only if the sequence $p_1,\dots,p_r$ is
quadratically independent.
\qed
\end{cor}

In particular, whether or not $\theta=\sum_{i=1}^rp_i\otimes p_i$
(with the $p_i$ linearly independent) is an extreme point of
$\Gram_V(f)$, depends only on the linear subspace
$U:=\spn(p_1,\dots,p_r)$, but not on~$f=\sum_{i=1}^rp_i^2$.

\begin{cor}\label{dimgramspec}%
Let $f\in A$, let $U\subset V$ be the linear subspace generated by
all $p\in V$ with $f-p^2\in\Sigma V^2$. Then
$$\dim\Gram_V(f)\>=\>\frac r2(r+1)-s$$
where $r=\dim(U)$ and $s=\dim(UU)$.
\qed
\end{cor}

%-------------------------------------------------------------------%

\section{Quadratically independent binary forms}

\begin{lab}
Let $k$ be a field, let $A$ be a (commutative) $k$-algebra. If
$U\subset A$ is a $k$-linear subspace, let $UU$ denote the linear
subspace of $A$ spanned by the products $pp'$ ($p,\,p'\in U$), as in
\ref{uu}. Assuming $\dim(U)=r<\infty$, we say that $U$ is
\emph{quadratically independent} if the natural multiplication map
$\sfS_2U\to A$ is injective, i.e.\ if $\dim(UU)=\choose{r+1}2$.
A sequence $p_1,\dots,p_r$ of elements of $A$ is \emph{quadratically
independent} if the $p_i$ are a linear basis of a quadratically
independent subspace $U$ of~$A$.
\end{lab}

We will prove the following general result for binary forms:

\begin{thm}\label{quadindepallg}%
Let $k$ be an infinite field, and let $d,\,r\ge1$ such that
$\choose{r+1}2\le2d+1$. Then there exists a sequence of $r$ binary
forms of degree~$d$ over $k$ that is quadratically independent.
\end{thm}

\begin{lab}
For the rest of this section write $A=k[x_1,x_2]=\bigoplus_{d\ge0}
A_d$, where $A_d$ is the space of binary forms of degree~$d$. Note
that $\dim(A_d)=d+1$. Clearly, the existence of a single
quadratically independent sequence of length~$r$ in $A_d$ implies
that the generic length~$r$ sequence in $A_d$ will be quadratically
independent. We can therefore assume that the field $k$ is
algebraically closed. (This assumption is only made to simplify
notation.)
\end{lab}

\begin{lab}
Our proof of Theorem \ref{quadindepallg}
proceeds by induction on $r\ge1$, the start being the case $r=1$ and
$d=0$ (which is obvious). So let $r\ge2$ in the sequel. By induction
there is a quadratically independent sequence $q_1,\dots,q_{r-1}$ in
$A_e$, where $e\ge0$ is minimal with $\choose r2\le2e+1$. Let $d\ge1$
be minimal with $\choose{r+1}2\le2d+1$. Given $z_1,\dots,z_m\in\P^1$
we put
$$W_d(z_1,\dots,z_m)\>:=\>\bigl\{f\in A_d\colon f(z_1)=\cdots=f(z_m)
=0\bigr\}.$$
Let $\infty\in\P^1$ be a fixed point, let $0\ne l\in A_1$ with
$l(\infty)=0$.
\end{lab}

\begin{lem}\label{lem1}%
Under these assumptions the following hold:
\begin{itemize}
\item[(a)]
For any linear subspace $U\subset W_d(\infty)$ and any $p\in U$, the
inequality
$$\dim(pA_d\cap UU)\>\ge\>\max\bigl\{\dim(U),\>\dim(UU)-d+1\bigr\}$$
holds.
\item[(b)]
There exists a subspace $U\subset W_d(\infty)$ with $\dim(U)=r-1$ and
$\dim(UU)=\choose r2$, together with a form $p\in U$, such that
equality holds in~(a).
\end{itemize}
\end{lem}

\begin{proof}
(a)
From $pU\subset pA_d\cap UU$ we get $\dim(pA_d\cap UU)\ge\dim(U)$.
Moreover $\dim(pA_d+UU)\le2d$ since $pA_d+UU\subset W_{2d}(\infty)$,
therefore $\dim(pA_d\cap UU)\ge(d+1)+\dim(UU)-2d=\dim(UU)-d+1$.

(b)
By induction we have a quadratically independent sequence
$q_1,\dots,q_{r-1}$ in $A_e$. Since $e<d$, the $r-1$ forms
$p_i:=l^{d-e}q_i$ ($1\le i\le r-1$) are in $W_d(\infty)$ and are
quadratically independent. Let $V=\spn(p_2,\dots,p_{r-1})$, we have
$\dim(VV)=\choose{r-1}2$. For sufficiently general $q\in
W_d(\infty)$ we claim that $qA_d\cap VV=\{0\}$ (if $r\le5$), resp.\
$qA_d\cap VV$ has codimension $d-1$ in $VV$ (if $r\ge5$). Indeed, if
$q$ has distinct zeros $z_1,\dots,z_{d-1},\infty$ in $\P^1$, we have
$qA_d\cap VV=W_{2d}(z_1,\dots,z_{d-1})\cap VV$. For general enough
choice of $q$, therefore, this intersection has codimension $d-1$ in
$VV$, resp.\ is zero if $\dim(VV)\le d-1$ (which happens precisely
for $r\le5$). We can therefore modify $p_1\in W_d(\infty)$ in such a
way that
$$\dim(p_1A_d\cap VV)\>=\>\begin{cases}0&r\le5,\\\choose{r-1}2-d+1&
r\ge5\end{cases}$$
holds and the sequence $p_1,\dots,p_{r-1}$ remains quadratically
independent. Writing $U:=\spn(p_1,\dots,p_{r-1})=kp_1\oplus V$ we
have $UU=p_1U\oplus VV$ since $U$ is quadratically independent.
Therefore
$$p_1A_d\cap UU\>=\>p_1U\oplus(p_1A_d\cap VV),$$
and this subspace has dimension $r-1$ (if $r\le5$) resp.\
$(r-1)+\choose{r-1}2-d+1=\choose r2-d+1$ (if $r\ge5$).
\end{proof}

\begin{lab}\label{setuplem2}%
According to Lemma \ref{lem1}, we can now fix a quadratically
independent subspace $U\subset W_d(\infty)$ with $\dim(U)=r-1$ and
such that
$$\dim(pA_d\cap UU)\>\ge\>\begin{cases}r-1&r\le5,\\\choose r2-d+1&
r\ge5\end{cases}$$
holds for all $p\in U$, with equality holding for $p$ sufficiently
general. We are going to show that we can extend $U$ to a
quadratically independent subspace of $A_d$ of dimension~$r$. Let
$\P_U$ resp.\ $\P_{A_d}$ denote the projective spaces associated to
the linear spaces $U$ resp.\ $A_d$, and consider the closed
subvariety
$$X\>:=\bigl\{([p],\,[q])\in\P_U\times\P_{A_d}\colon pq\in
UU\bigr\}$$
of $\P_U\times\P_{A_d}$. (Here we write $[p]$ for the element in
$\P_U$ represented by $0\ne p\in U$, and similarly $[q]$ for
$0\ne q\in A_d$.) Let $\pi_1\colon X\to\P_U$ and $\pi_2\colon
X\to\P_{A_d}$ denote the projections onto the two components.

Let $\epsilon\in\{0,1\}$ be defined by $2d+1=\choose{r+1}2+\epsilon$.
We can calculate the dimension of~$X$:
\end{lab}

\begin{lem}\label{lem2}%
$\dim(X)=d-1$ if $r\le5$, and $\dim(X)=d-1-\epsilon$ if $r\ge5$.
\end{lem}

\begin{proof}
Clearly $\pi_1$ is surjective since $([p],[p])\in X$ for $0\ne p\in
U$. For $0\ne p\in U$, the fibre $\pi_1^{-1}([p])$ has (projective)
dimension $\dim(pA_d\cap UU)-1$. From \ref{setuplem2} we therefore
see that the generic fibre of $\pi_1$ has dimension $r-2$ (if
$r\le5$) resp.\ $\choose r2-d$ (if $r\ge5$). It follows that
$\dim(X)=2r-4=d-1$ if $r\le5$, resp.\ $\dim(X)=(r-2)+\choose r2-d=
\choose{r+1}2-d-2=(2d+1-\epsilon)-d-2=d-1-\epsilon$ if $r\ge5$.
\end{proof}

\begin{lab}
In particular, $\dim(X)<\dim(\P_{A_d})$. For generically chosen
$q\in A_d$, therefore, we have $\pi_2^{-1}([q])=\emptyset$, which
means $qU\cap UU=\{0\}$. In particular there is such $q\in A_d$ with
$q(\infty)\ne0$. Since $qU\oplus UU\subset W_{2d}(\infty)$ and
$q^2\notin W_{2d}(\infty)$, we see that the $r$-dimensional subspace
$U+kq$ of $A_d$ is quadratically independent. This completes the
induction step, and thereby the proof of Theorem \ref{quadindepallg}.
\qed
\end{lab}

%-------------------------------------------------------------------%

\section{Pataki range for Gram spectrahedra of binary forms}

\begin{lab}
Let $n\ge2$ be fixed. For $d\ge1$, $\R[x]_d$ denotes the space of
forms of degree $d$ in $\R[x]=\R[x_1,\dots,x_n]$. We write $N_d=
\dim\R[x]_d=\choose{n+d-1}d$. Let $\Sigma_{2d}\subset\R[x]_{2d}$
denote the sums of squares cone, i.e.\ $\Sigma_{2d}=\Sigma\R[x]_d^2$.
For $f\in\Sigma_{2d}$ let $\Gram(f)$ be the (full) Gram spectrahedron
of $f$, i.e.\ $\Gram(f):=\Gram_V(f)$ with $V:=\R[x]_d$. Since
$\Gram(f)=\mu^{-1}(f)\cap\sfS_2^\plus V$ and
$$\dim\mu^{-1}(f)\>=\>\dim(\sfS_2V)-\dim(VV)\>=\>\choose{N_d+1}2
-N_{2d},$$
the Pataki interval (\ref{dfnpatrang}) for $\Gram(f)$ is
characterized by the inequalities
$$N_{2d}+\choose{N_d-r+1}2\>\le\>\choose{N_d+1}2\text{ \ and \ }
\choose{r+1}2\>\le\>N_{2d}.$$
For $f\in\interior(\Sigma_{2d})$ we have $\dim\Gram(f)=
\dim\mu^{-1}(f)=\choose{N_d+1}2-N_{2d}$. In the case $n=2$ of binary
forms this means $\dim\Gram(f)=\choose{d+2}2-(2d+1)=\choose d2$ for
$f\in\interior(\Sigma_{2d})$,
and the Pataki range is described by the inequalities $r\ge2$ and
$\choose{r+1}2\le2d+1$.
\end{lab}

In what follows we always work with binary forms, i.e.\ $n=2$ and
$\R[x]=\R[x_1,x_2]$. From Theorem \ref{quadindepallg} we get:

\begin{cor}\label{quadindep}%
Let $d\ge1$ and $r\ge0$ such that $\choose{r+1}2\le2d+1$. The set of
quadratically independent $r$-tuples $(p_1,\dots,p_r)$ in
$(\R[x_1,x_2]_d)^r$ is open and dense.
\end{cor}

Here is our first main result on extreme points of Gram spectrahedra:

\begin{thm}\label{main}%
For any given $d\ge1$, there is an open dense set of psd binary forms
$f$ of degree~$2d$ for which the Gram spectrahedron $\Gram(f)$ has
extreme points of all ranks in the Pataki interval.
\end{thm}

This gives an affirmative answer to Question 4.2 from \cite{cpsv}.
Note that $\Gram(f)$ has dimension $\choose d2$ for general
$f\in\Sigma_{2d}$, so the dimensions of these spectrahedra are
arbitrarily large.

\begin{proof}[Proof of Theorem \ref{main}]
Let $k\ge1$ be the largest integer with $\choose{k+1}2\le2d+1$, so
the Pataki interval for Gram spectrahedra of degree $2d$ forms is
$\{2,\,3,\,\dots,\,k\}$. Fix $r\in\{2,\,3,\,\dots,\,k\}$, and let
$W_r\subset(\R[x]_d)^r$ be the set of all quadratically independent
$r$-tuples $(p_1,\dots,p_r)$ of forms.
By Corollary \ref{quadindep}, the set $W_r$ is open and dense in
$(\R[x]_d)^r$. Let
$$S_r\>:=\>\bigl\{p_1^2+\cdots+p_r^2\colon(p_1,\dots,p_r)\in W_r
\bigr\}.$$
Since every psd form in $\R[x]$ is a sum of two squares, the set
$S_r$ is a dense semialgebraic subset of $\Sigma_{2d}$.
Whenever $(p_1,\dots,p_r)\in W_r$, if we put $f:=\sum_{i=1}^rp_i^2$,
the symmetric tensor $\sum_{i=1}^rp_i\otimes p_i$ is an extreme point
of $\Gram(f)$ of rank~$r$ (Corollary \ref{chaktextpts}). Therefore
every $f\in S_r$ has a rank~$r$ extreme point in its Gram
spectrahedron. It now suffices to consider the intersection
$S:=\bigcap_{r=2}^kS_r$. Then $S$ is a dense semialgebraic subset of
$\Sigma_{2d}$ since $\dim(\Sigma_{2d}\setminus S)<\dim(\Sigma_{2d})$.
And for every $f\in S$, the Gram spectrahedron of $f$ has extreme
points of all ranks in the Pataki interval.
\end{proof}

We can also determine the dimensions of the sets of extreme points of
a fixed rank, for suitably general $f$. To have a short notation, let
us write $\Ex_r(f)$ for the (semialgebraic) set of all extreme points
of $\Gram(f)$ of rank~$r$.

\begin{cor}\label{dimexr}%
Let $d\ge1$. There is an open dense subset $U$ of $\Sigma_{2d}$ such
that, for every $f\in U$ and every $r$ in the Pataki range, we have
$$\dim\Ex_r(f)\>=\>\frac12(r-2)(2d-r+1).$$
\end{cor}

\begin{proof}
Let $r$ be in the Pataki range. Using notation from the previous
proof, consider the sum of squares map $\sigma\colon W_r\to
\R[x]_{2d}$, $(p_1,\dots,p_r)\mapsto\sum_{i=1}^rp_i^2$. Its image is
dense in $\Sigma_{2d}$. It follows from local triviality of
semialgebraic maps (Hardt's theorem, see e.g.\ \cite{bcr} Theorem
9.3.2)
that, for every $f$ in an open dense set $U_r\subset\Sigma_{2d}$,
the fibre $\sigma^{-1}(f)$ has dimension $r(d+1)-(2d+1)$.
The orthogonal group $O(r)$ has dimension $\choose r2$. It acts on
the fibre $\sigma^{-1}(f)$ with trivial stabilizer subgroups,
and the orbits are precisely the extreme points of $\Gram(f)$ of rank
$r$. So we get
$$\dim\Ex_r(f)\>=\>r(d+1)-(2d+1)-\choose r2\>=\>
\frac12(r-2)(2d-r+1)$$
for every $f\in U_r$. Take $U$ to be the intersection of the sets
$U_r$ for all $r$ in the Pataki range, to get the desired conclusion.
\end{proof}

\begin{rem}
At least for general positive $f$ of degree $\ge12$, the boundary of
$\Gram(f)$ is a union of positive dimensional faces.
This is reflected by the fact that, for $2d\ge8$ and any $r$ in the
Pataki range, the number $\frac12(r-2)(2d-r+1)$ from Corollary
\ref{dimexr} is smaller than the dimension of the boundary of
$\Gram(f)$, which is $\choose d2-1$, for general $f\in\Sigma_{2d}$.
\end{rem}

%-------------------------------------------------------------------%

\section{Edges between extreme points of rank two}

\begin{lab}
We keep considering binary forms, so we work in $\R[x]=\R[x_1,x_2]$.
Let $f\in\Sigma_{2d}$. Recall (\cite{clr} Example 2.13, \cite{cpsv}
Proposition 4.1) how Gram
tensors $\theta\in\Gram(f)$ of rank $\le2$ correspond to product
decompositions $f=g\ol g$ with $g\in\C[x]$, where $\ol g$ is the form
that is coefficient-wise complex conjugate to $g$. Any $\theta\in
\Gram(f)$ with $\rk(\theta)\le2$ has the form $\theta=p\otimes p+
q\otimes q$ where $p,\,q\in\R[x]$ satisfy $f=p^2+q^2=(p+iq)(p-iq)$.
Conversely, a factorization $f=g\ol g$ with $g\in\C[x]$ gives a Gram
tensor $\theta=p\otimes p+q\otimes q$ of $f$, namely $p=\frac12
(g+\ol g)$ and $q=\frac1{2i}(g-\ol g)\in\R[x]$. Two factorizations
$f=g\ol g=h\ol h$ give the same Gram tensor of $f$ if and only if $h$
is a scalar multiple of $g$ or $\ol g$. In particular, if we assume
that $f$ has no multiple complex roots, we see that $f$ has (no Gram
tensors of rank one and) precisely $2^{d-1}$ Gram tensors of rank
two. All of them are extreme points of $\Gram(f)$.
\end{lab}

\begin{lab}
When $g$ has only real zeros, $\Gram(f)\cong\Gram(fg^2)$ naturally.
Hence we discuss $\Gram(f)$ for strictly positive $f$ only. Let
$d\ge1$, let $f\in\Sigma_{2d}$ be strictly positive, and let us first
consider the cases of very small degree. If $d=1$ then $\Gram(f)$ is
a single point of rank two. If $d=2$ then $\Gram(f)$ is a
nondegenerate interval, the relative interior of which consists of
points of rank~$3$. If $f$ has simple roots, both end points have
rank~$2$. Otherwise $f$ is a square, and one end point has rank~$1$,
the other has rank~$2$.

The case $d=3$ is covered in the next result (see also \cite{cpsv}
Section 4.2):
\end{lab}

\begin{prop}
Let $f\in\Sigma_6$ be strictly positive. Then $\dim\Gram(f)=3$, and
the points in $\relint\Gram(f)$ have rank~$4$. Moreover,
\begin{itemize}
\item[(a)]
$\Gram(f)$ has no faces of dimension $1$ or $2$,
\item[(b)]
$\Gram(f)$ has $4$, $3$ or $2$ extreme points of rank $\le2$,
\item[(c)]
all other extreme points have rank~$3$.
\end{itemize}
\end{prop}

\begin{proof}
The extreme
points of rank $\le2$ correspond to complex factorizations
$f=p\ol p$. Depending on whether $f$ has six, four or two different
roots, there are four, three or two essentially different such
factorizations. The corresponding psd Gram tensors have rank two
except when $f$ is a square, i.e.\ has only two different roots; then
one of the Gram tensors has rank one.
If $\Gram(f)$ had a proper face of positive dimension, its rank would
have to be~$3$. To prove (a) it therefore suffices to show that, for
any two extreme points $\theta\ne\theta'$ of rank $\le2$, the segment
$[\theta,\theta']$ meets the interior of $\Gram(f)$. Let
$f=p\ol p=q\ol q$ be the two factorizations corresponding to $\theta$
and $\theta'$. We can assume $p=gh$, $q=g\ol h$ with
$$g\>=\>(x-a_1)(x-a_2),\ h\>=\>x-a_3$$
and $\{a_1,a_2,a_3\}\cap\{\ol a_1,\ol a_2,\ol a_3\}=\emptyset$.
For the supporting face $F$ of $\frac12(\theta+\theta')$ we
have
$$\scrU(F)\>=\>\spn(gh,\>g\ol h,\>\ol gh,\>\ol{gh})$$
Calculating the determinant gives
$$(a_1-\ol a_1)(a_1-\ol a_2)(a_2-\ol a_1)(a_2-\ol a_2)
(a_3-\ol a_3)^2\>\ne\>0$$
This means that $\frac12(\theta+\theta')$ has rank $4$, and hence
lies in the interior of $\Gram(f)$.
\end{proof}

When the positive sextic $f$ is general, the algebraic boundary of
$\Gram(f)$ is a Kummer surface, see \cite{orsv} Section~5 and
\cite{cpsv} Section 4.2. In this case, assertion (a) also follows
from the fact that a Kummer surface in $\P^3$ does not contain a
line.

Now we are interested in arbitrary degrees. Let $f\in\R[x]_{2d}$ be a
sufficiently general positive form. We ask: For which pairs
$\theta\ne\theta'$ in $\Ex_2(f)$ is the line segment
$[\theta,\theta']$ an edge of $\Gram(f)$, i.e.\ a one-dimensional
face?

\begin{thm}\label{edges}%
Let $d\ge4$. For all forms $f$ in an open dense subset of
$\Sigma_{2d}$, the following is true:
\begin{itemize}
\item[(a)]
$d=4$: For each of the $\choose82=28$ pairs $\theta\ne\theta'$ in
$\Ex_2(f)$, the interval $[\theta,\theta']$ is contained in the
boundary of $\Gram(f)$. For precisely $16$ of these pairs,
$[\theta,\theta']$ is a face of $\Gram(f)$. These $16$ edges form a
graph isomorphic to $K_{4,4}$, the complete bipartite graph on two
sets of four points each.
\item[(b)]
$d\ge5$: For any two $\theta\ne\theta'$ in $\Ex_2(f)$, the
line segment $[\theta,\theta']$ is a face of $\Gram(f)$.
\end{itemize}
\end{thm}

\begin{lab}\label{edgesetup}%
Let $f=p\ol p=q\ol q$ be complex factorizations of $f$ that
correspond to $\theta$ and $\theta'$, respectively. The supporting
face $F$ of $[\theta,\theta']$ therefore has $\scrU(F)_\C=
\spn(p,\ol p,q,\ol q)\subset\C[x]_d$, and $\dim(F)$ is the number of
quadratic relations between $p,\ol p,\,q$ and $\ol q$. We can split
$p=gh$ into two nontrivial complex factors in such a way that
$\theta'$ corresponds to the factorization $f=q\ol q$ with
$q=g\ol h$. Thus
$$\scrU(F)_\C\>=\>\spn(gh,\>g\ol h,\>\ol gh,\>\ol{gh}).$$
For general $f$ we have $\dim\scrU(F)=4$.
Assuming this, $[\theta,\theta']$ is an edge of $\Gram(f)$ if and
only if there is only one quadratic relation between $p=gh$,
$\ol p=\ol{gh}$, $q=g\ol h$ and $\ol q=\ol gh$, i.e.\ if and only if
the nine products
$$g^{a_1}\,\ol g^{a_2}\,h^{b_2}\,\ol h^{b_2},\quad a_i,\,b_i\ge0,\
a_1+a_2=b_1+b_2=2\eqno(*)$$
are linearly independent. (To be sure, there always is one quadratic
relation between $p,\,\ol p,\,q$ and $\ol q$, namely
$p\ol p=q\ol q$.)

The key case for Theorem \ref{edges} is $d=4$.
It is made more explicit in the next two lemmas:
\end{lab}

\begin{lem}\label{31case}%
Let $g_1,\,g_2\in\C[x]$ have degree~$3$, let $h_1,\,h_2\in\C[x]$ have
degree~$1$. Then the nine octic forms
$$g_1^{a_1}g_2^{a_2}h_1^{b_1}h_2^{b_2},\quad a_i,\,b_i\ge0,\
a_1+a_2=b_1+b_2=2$$
are linearly independent if (and only if) $\gcd(g_1,g_2)=
\gcd(h_1,h_2)=1$.
\end{lem}

\begin{lem}\label{22case}%
For arbitrary $g_1,\,g_2,\,h_1,\,h_2\in\C[x]$ of degree~$2$, the nine
octic forms
$$g_1^{a_1}g_2^{a_2}h_1^{b_1}h_2^{b_2},\quad a_i,\,b_i\ge0,\
a_1+a_2=b_1+b_2=2$$
are linearly dependent.
\end{lem}

\begin{cor}\label{ge5case}%
Let $g_1,\,g_2,\,h_1,\,h_2\in\C[x]$ with $\deg(g_1)=\deg(g_2)=
\delta\ge1$, $\deg(h_1)=\deg(h_2)=\epsilon\ge1$ and
$\delta+\epsilon\ge5$. If $g_1,g_2,h_1,h_2$ are chosen generically,
the nine forms
$$g_1^{a_1}g_2^{a_2}h_1^{b_1}h_2^{b_2},\quad a_i,\,b_i\ge0,\
a_1+a_2=b_1+b_2=2$$
(of degree $2(\delta+\epsilon)$) are linearly independent.
\end{cor}

\begin{lab}
Before establishing \ref{31case}, \ref{22case} and \ref{ge5case}, we
show how these imply Theorem \ref{edges}. First let $d=4$, let
$f\in\Sigma_8$ have simple complex zeros, and let $f=p\ol p=q\ol q$
be two nontrivial factorizations corresponding to extreme points
$\theta\ne\theta'$ in $\Gram(f)$ (c.f.\ \ref{edgesetup}). Since
$\rk(\theta+\theta')\le\rk(\theta)+\rk(\theta')=4$, it is obvious
that $[\theta,\theta']$ is contained in the boundary of $\Gram(f)$.
Write $p=g_1g_2$ and $q=g_1\ol g_2$ as in \ref{edgesetup}. If
$\deg(g_1)=\deg(g_2)=2$, the nine forms $(*)$ (see \ref{edgesetup})
are linearly dependent by Lemma \ref{22case}, and so
$[\theta,\theta']$ is not an edge. Otherwise $\{\deg(g_1),\,
\deg(g_2)\}=\{1,3\}$. By Lemma \ref{31case}, therefore, the nine
forms $(*)$ are linearly independent, and so $[\theta,\theta']$ is an
edge.

This proves the $d=4$ case of Theorem \ref{edges}. Indeed, the eight
points of $\Ex_2(f)$, corresponding to the eight essentially
different factorizations $f=p\ol p$, decompose into two subclasses of
four points each, where two different factorizations
$f=p\ol p=q\ol q$ belong to the same subclass if and only if $p$ and
$q$ have precisely two roots in common.

If $d\ge5$, if $f\in\Sigma_{2d}$ is sufficiently general, and if
$f=p\ol p=q\ol q$ are two factorizations belonging to $\theta\ne
\theta'$, Corollary \ref{ge5case} shows that $(*)$ are linearly
independent, whence $[\theta,\theta']$ is an edge.
\end{lab}

\begin{proof}[Proof of Lemma \ref{31case}]
It is obvious that $\gcd(g_1,g_2)=\gcd(h_1,h_2)=1$ are necessary for
the nine octics to be linearly independent.
For the converse assume these conditions, and consider the ideals
$I=\idl{g_1,g_2}$ and $J=\idl{h_1,h_2}$ in $A=k[x]$.
We have to prove
$(I^2J^2)_8=A_8$. Now $\gcd(h_1,h_2)=1$ implies $J_1=A_1$ and hence
$(J^2)_2=A_2$. So $(I^2J^2)_8$ contains $(I^2)_6A_2=(I^2)_8$, and it
is enough to prove $(I^2)_8=A_8$. The ideal $I$ is a complete
intersection since $\gcd(g_1,g_2)=1$, hence a Gorenstein ideal of
socle degree~$4$. So $I_5=A_5$,
and so $(I^2)_8$ contains $I_5I_3=A_5I_3=I_8=A_8$.
\end{proof}

\begin{proof}[Proof of Lemma \ref{22case}]
Unfortunately, we have no better argument than a brute force
computation: For $g_i,\,h_i$ with general coefficients, the
corresponding $9\times9$ determinant vanishes identically.
\end{proof}

\begin{proof}[Proof of Corollary \ref{ge5case}]
It suffices to prove the assertion for one specific choice of the
$g_i$ and $h_i$. We can assume $\delta\ge3$.
Let $G_1,\,G_2,\,H_1,\,H_2$ satisfy $\deg(G_i)=3$, $\deg(H_i)=1$ and
$\gcd(G_1,G_2)=\gcd(H_1,H_2)=1$, and let $l\ne0$ be any linear form.
Then by Lemma \ref{31case}, the assertion of \ref{ge5case} is true
for $g_i:=G_i\ell^{\delta-3}$, $h_i:=H_i\ell^{\epsilon-1}$, $i=1,2$.
\end{proof}

%===================================================================%


\begin{thebibliography}{mmm}

\bibitem{bpsv}
G. Blekherman, D. Plaumann, R. Sinn, C. Vinzant:
Low-rank sum-of-squares representations on varieties of minimal
degree.
Int.\ Math.\ Res.\ Notes\ 2017, 1--22.

\bibitem{bsv}
G. Blekherman, G\,G. Smith, M. Velasco:
Sums of squares and varieties of minimal degree.
J.~Am.\ Math.\ Soc.\ \textbf{29}, 893--913 (2016).

\bibitem{bcr}
J. Bochnak, M. Coste, M.-F. Roy:
\emph{Real Algebraic Geometry}.
Erg.\ Math.\ Grenzgeb.\ (3) \textbf{36}, Springer, Berlin, 1998.

\bibitem{clr}
M.\,D. Choi, T.\,Y. Lam, B. Reznick:
Sums of squares of real polynomials.
In $K$-Theory and Algebraic Geometry: Connections with Quadratic
Forms and Division Algebrs, Proc.\ Sym.\ Pure Math.\ \textbf{58.2},
B.~Jacob and A.~Rosenberg (eds), AMS, Providence RI, 1995, pp.\
103--126.

\bibitem{cpsv}
L. Chua, D. Plaumann, R. Sinn, C. Vinzant:
Gram spectrahedra.
In: Ordered Algebraic Structures and Related Topics,
F.~Broglia et al (eds), Contemp.\ Math.\ \textbf{697},
Am.\ Math.\ Soc., Providence, RI, 2017, pp.\ 81--105.

\bibitem{hi}
D. Hilbert:
\"Uber die Darstellung definiter Formen als Summe von
Formenquadraten.
Math.\ Ann.\ \textbf{32}, 342--350 (1888).

\bibitem{ne}
T.~Netzer:
Stability of quadratic modules.
Manuscr.\ math.\ \textbf{129}, 251--271 (2009).

\bibitem{nrs}
J.~Nie, K.~Ranestad, B.~Sturmfels:
The algebraic degree of semidefinite programming.
Math.\ Program., Ser.~A, \textbf{122}, 379--405 (2010).

\bibitem{orsv}
J.\,C. Ottem, K. Ranestad, B. Sturmfels, C. Vinzant:
Quartic spectrahedra.
Math.\ Program., Ser.~B, \textbf{151}, 585--612 (2015).

\bibitem{pa}
G.~Pataki:
The geometry of semidefinite programming.
In \cite{wsv}, pp.\ 29--65.

\bibitem{psv}
D. Plaumann, B. Sturmfels, C. Vinzant:
Quartic curves and their bitangents.
J.~Symbolic Comput.\ \textbf{46}, 712--733 (2011).

\bibitem{prss}
V. Powers, B. Reznick, C. Scheiderer, F. Sottile:
A new approach to Hilbert's theorem on ternary quartics.
C.\,R.~Acad.\ Sci.\ Paris, ser.~I, \textbf{339}, 617--620 (2004).

\bibitem{rg}
M. Ramana, A.\,J. Goldman:
Some geometric results in semidefinite programming.
J.~Global Optim.\ \textbf{7}, 33--50 (1995).

\bibitem{sch:pyth}
C. Scheiderer:
Sum of squares length of real forms.
Math.~Z.\ \textbf{286}, 559--570 (2017).

\bibitem{wsv}
H.~Wolkowicz, R.~Saigal, L.~Vandenberghe (eds.):
\emph{Handbook of Semidefinite Programming. Theory, Algorithms, and
Applications}.
Kluwer, Boston, 2000.

\end{thebibliography}
\end{document}